\documentclass[a4paper]{amsart}

\usepackage{hyperref}
\usepackage{cleveref}
\usepackage{amsmath,amsthm}
\usepackage{amsfonts,amssymb,amscd}
\usepackage{graphicx}
\usepackage[margin=3cm]{geometry}
\usepackage{float}
\usepackage{cases}
\usepackage[utf8]{inputenc}

\numberwithin{equation}{section}

\newcommand*\diff{\mathop{}\!\mathrm{d}}

\DeclareMathOperator{\spann}{\mathrm{span\,}}
\newcommand{\circlenum}[1]{\raisebox{.5pt}{\textcircled{\raisebox{-.9pt} {#1}}}}

\newtheorem{theo}{Theorem}
\newtheorem{lemm}{Lemma}[section]

\newtheorem{prop}{Proposition}[section]
\newtheorem{fact}{Fact}[section]
\newtheorem*{coro}{Corollary}

\theoremstyle{definition}
\numberwithin{defi}{section}

\theoremstyle{remark}
\newtheorem*{rema}{Remark}

\title[Minimal surfaces with planar curvature lines]{Deformation of minimal surfaces with planar curvature lines}
\author{Joseph Cho}
\author{Yuta Ogata}

\date{}

\address{\newline Department of mathematics, Faculty of science\newline Kobe university\newline Rokkodai-cho 1-1, Nada-ku, Kobe-shi, Hyogo-ken, 657-8501\newline Japan}
\email{jcho@stu.kobe-u.ac.jp}
\email{ogata@math.kobe-u.ac.jp}
\begin{document}

\begin{abstract}
Minimal surfaces with planar curvature lines are classical geometric objects, having been studied since the late 19th century. In this paper, we revisit the subject from a different point of view. After calculating their metric functions using an analytical method, we recover the Weierstrass data, and give clean parametrizations for these surfaces. Then, using these parametrizations, we show that there exists a single continuous deformation between all minimal surfaces with planar curvature lines. In the process, we establish the existence of axial directions for these surfaces.
\end{abstract}

\maketitle
\normalsize

\section{Introduction}
The study of minimal surfaces with planar curvature lines is a classical subject, having been studied by Bonnet, Enneper, and Eisenhart in the late 19th century as recorded in \cite{bonnet1855}, \cite{eisenhart1909}, and \cite{enneper1878}, respectively. The subject was further studied by Nitsche in \cite{nitsche1989}, where he gave a full classification of such surfaces. Nitsche showed that families of planar curvature lines transform into orthogonal families of circles on the unit sphere $\mathbb{S}^2$ under the Gauss map, and by analyzing orthogonal systems of circles, he recovered the data for the following well-known theorem given by Weierstrass in \cite{weierstrass1866}.
\begin{fact}[Weierstrass representation theorem]
Any minimal surface $X : \Sigma \subset \mathbb{C} \to \mathbb{R}^3$ can be locally represented as
$$X = \,\mathrm{Re}\int \left(1 - h^2, i(1 + h^2), 2 h\right)\eta\,\diff z$$
over a simply-connected domain $\Sigma$ on which $h$ is meromorphic, while $\eta$ and $h^2\eta$ are holomorphic.
\end{fact}
\noindent Using the Weierstrass data $(h, \eta \diff z)$ and the representation, Nitsche classified different types of minimal surfaces with planar curvature lines, stated here along with their respective Weierstrass data.
\begin{fact}
A minimal surface in Euclidean space $\mathbb{R}^3$ with planar curvatures lines must be a piece of one, and only one, of
	\begin{itemize}
		\item plane $(0, 1 \diff z)$,
		\item catenoid $(e^z, e^{-z} \diff z)$,
		\item Enneper surface $(z, 1 \diff z)$, or
		\item a surface in the Bonnet family $\{(e^z + t, e^{-z} \diff z), t > 0\}$
	\end{itemize}
up to isometries and homotheties of $\mathbb{R}^3$.
\end{fact}

\begin{figure}[H]
	\includegraphics{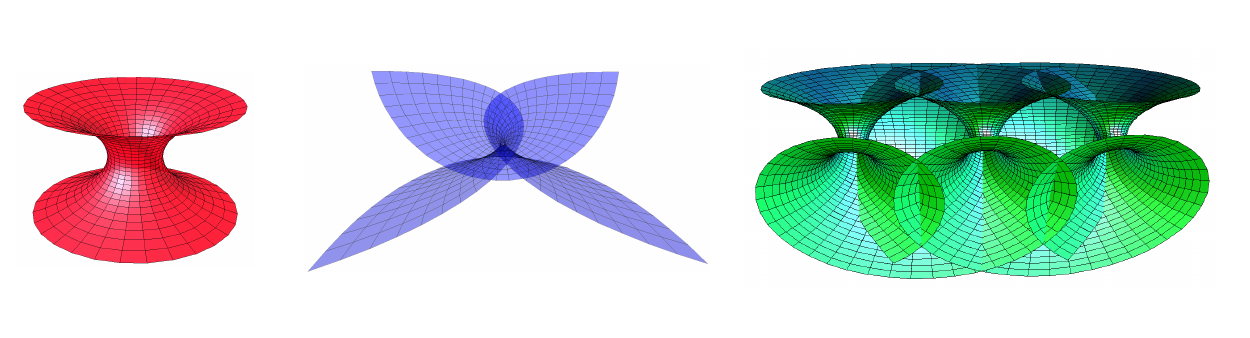}
	\caption{Examples of non-planar minimal surfaces with planar curvature lines: catenoid, Enneper surface, and a surface from the Bonnet family.}
	 \label{fig:surfaces}
\end{figure}

In fact, planes can also be considered as spheres with infinite radius. Thus, minimal surfaces with planar curvature lines can be thought of as a special case of minimal surfaces with spherical curvature lines. There are many works focused on such surfaces, including \cite{dobriner1887} and \cite{wente1992}.

On the other hand, Thomsen studied surfaces with zero mean curvature which are also affine minimal. In \cite{thomsen1923}, he classified such surfaces using the fact that the families of asymptotic lines of these surfaces transformed into orthogonal families of circles on the unit sphere under the Gauss map. He also mentioned that since the asymptotic lines correspond to the curvature lines of the conjugate surface, minimal surfaces that are also affine minimal are conjugate surfaces of minimal surfaces with planar curvature lines. Building on the result of Thomsen, Schaal showed that the plane and the Enneper surface can be attained as a limit of Thomsen surfaces in \cite{schaal1973}, and Barthel, Volkmer, and Haubitz were able join these surfaces into a one-parameter family of surfaces using the analytical approach in \cite{barthel1980}.

It is also possible to attain a deformation between minimal surfaces with planar curvature lines using Goursat transformation introduced in \cite{goursat1887} and \cite{goursat1887-1}. A Goursat transformation transforms minimal surfaces to minimal surfaces. In addition, the transformation not only maps curvature lines to curvature lines but also keeps the planar the curvature line condition as seen in \cite{hertrich-jeromin2001}, \cite{mladenov1999}, or \cite{pabel1994}. However, one may not transform a minimal surface into a plane via Goursat transformation alone.

Meanwhile, non-zero constant mean curvature (CMC) surfaces with planar curvature lines also has been studied extensively, as in \cite{abresch1987}, \cite{walter1987}, or \cite{wente1986}. In particular, Wente constructed non-trivial examples of compact CMC surfaces, called Wente tori in \cite{wente1986}. Then in \cite{abresch1987}, Abresch gave a classification of all CMC surfaces with planar curvature lines including the Wente tori and cylindrical ended bubbletons, by solving a system of partial differential equations. Similarly, in \cite{walter1987} Walter considered explicit parametrizations of Wente tori, by showing the existence a notion of axes.

In this paper, to classify minimal surfaces in $\mathbb{R}^3$ with planar curvature lines, we propose an alternate method to using orthogonal systems of circles. In particular, we utilize the method analogous to the approach taken in \cite{abresch1987}, \cite{barthel1980}, \cite{walter1987}, and \cite{wente1992}, modified for the subject at hand. In \Cref{sec:minimal}, we first obtain and solve system of partial differential equations describing the metric function, similar to the method used in \cite{abresch1987} and \cite{barthel1980}. Then, following \cite{walter1987} and \cite{wente1992}, we prove the existence of axial directions of these surfaces, using them to recover the Weierstrass data, and ultimately their parametrizations. In \Cref{sec:deformation}, we use the results from the previous section to obtain a single-parameter deformation of all minimal surfaces preserving the planar curvature line condition. Finally as the main result, we state the classification, parametrization, and deformation of all minimal surfaces with planar curvature lines (see Theorem \ref{theo:min_main} and \Cref{fig:deformation1}).

\section{Minimal Surfaces with Planar Curvature Lines in $\mathbb{R}^3$} \label{sec:minimal}
We wish to classify minimal surfaces with planar curvature lines and obtain their parametrizations by calculating the Weierstrass data. However, applying the planar curvature line condition directly to the Weierstrass data involves heavy calculation in a complex-analysis setting. Therefore, from the zero mean curvature condition and planar curvature line condition, we obtain a system of partial differential equations for the metric function. Then by applying the method analogous to the one employed by Abresch in \cite{abresch1987}, we solve the system of partial differential equations by transforming it into a system of ordinary differential equations. Finally, using explicit solutions for the metric function, we recover the Weierstrass data and the parametrization by calculating the unit normal vector, through an analogous approach to the one taken by Walter in \cite{walter1987}.

\subsection{Minimal surface theory.}\label{subsec:min_Abresch}
Let $\Sigma \subset \mathbb{R}^2$ be a simply-connected domain with coordinates $(u,v)$, and let $X: \Sigma \to \mathbb{R}^3$ be a conformally immersed surface. Since $X(u,v)$ is conformal,
$$\diff s^2 = e^{2\omega}(\diff u^2+ \diff v^2)$$
for some $\omega : \Sigma \to \mathbb{R}$. We choose the unit normal vector field $N:\Sigma\longrightarrow
\mathbb{S}^2$ of $X$, and then the mean curvature $H$ and Hopf differential $Q$ are
$$ H:=\frac{1}{2e^{2\omega}} \langle X_{uu}+X_{vv},N\rangle \quad\text{and}\quad  Q:=\frac{1}{4}\langle X_{uu}-2iX_{uv}-X_{vv},N\rangle. $$
Since we are interested in minimal surfaces, we let $H = 0$ and calculate the Gauss-Weingarten equations,
$$
	\begin{cases}
       	 	X_{uu} = \omega_u X_u -\omega_v X_v +(Q+\bar{Q})N\\
        		X_{vv} = -\omega_u X_u +\omega_v X_v -(Q+\bar{Q})N\\
        		X_{uv} = \omega_v X_u +\omega_u X_v +i(Q-\bar{Q})N\\
        		N_u = -e^{-2\omega}(Q+\bar{Q}) X_u-ie^{-2\omega}(Q-\bar{Q}) X_v\\
        		N_v = -ie^{-2\omega}(Q-\bar{Q}) X_u+e^{-2\omega}(Q+\bar{Q}) X_v,
	\end{cases}
$$
and the Gauss-Codazzi equation,
$$
	\Delta \omega - 4Q\bar{Q}e^{-2\omega} = 0 \quad\text{and}\quad Q_{\bar{z}}=0
$$
for $z:=u+iv$. Note that the Gauss-Codazzi equation is equivalent to the Hopf differential factor $Q$ being holomorphic. Moreover, the Gauss-Codazzi equation is invariant under the deformation $Q\mapsto\lambda^{-2}Q$ for $\lambda\in\mathbb{S}^1\subset\mathbb{C}$. In fact, when $X(u,v)$ is a minimal surface in $\mathbb{R}^{3}$, $\lambda\in\mathbb{S}^1$ allows us to create a single-parameter family of minimal surfaces $X^{\lambda}(u,v)$ associated to $X(u,v)$, called the \emph {associated family}. In particular if $\lambda^{-2} = i$, then the new surface is called the \emph{conjugate} surface of $X$.

Since the families of curvature lines of a plane are trivially planar, we may assume that $X(u,v)$ is not totally umbilic, and that $(u,v)$ are conformal curvature line (or isothermic) coordinates. Then we can normalize the Hopf differential such that $Q=-\frac{1}{2}$, and the relevant Gauss-Weingarten equations become
\begin{equation}\label{eqn:min_gaussW}
	\begin{cases}
       	 	X_{uu} = \omega_u X_u - \omega_v X_v - N\\
        		X_{vv} = - \omega_u X_u + \omega_v X_v + N\\
        		X_{uv} = \omega_v X_u + \omega_u X_v\\
        		N_u = e^{-2\omega} X_u\\
        		N_v = - e^{-2\omega} X_v
	\end{cases}
\end{equation}
where $k_1 = -e^{-2\omega}$ and $k_2 = e^{-2\omega}$ are the principle curvatures of $X$. Furthermore, the Gauss equation becomes the following Liouville equation:
$$\Delta \omega - e^{-2\omega} = 0.$$

On the other hand, as the following lemma shows, we may attain an additional partial differential equation regarding $\omega$ from the planar curvature line condition, allowing us to solve for $\omega$.
\begin{lemm}\label{lemm:min_planarCondition}
For non-planar umbilic-free minimal surfaces with isothermic coordinates $(u,v)$, the following statements are equivalent:		
	\begin{enumerate}
		\item $u$-curvature lines are planar.
		\item $v$-curvature lines are planar.
		\item $\omega_{uv} + \omega_u \omega_v = 0$.
	\end{enumerate}
\end{lemm}
\begin{proof}
Since $(u,v)$ is an isothermic coordinate, $u$-curvature lines are planar if and only if
$$\det(X_u, X_{uu}, X_{uuu})=0.$$
However, from \eqref{eqn:min_gaussW},
$$X_{uuu} = (\omega_{uu} + \omega_u^2 - \omega_v^2 - e^{-2\omega}) X_u + (-2\omega_u \omega_v - \omega_{uv}) X_v - \omega_u N.$$
Therefore, a $u$-curvature line is planar if and only if
$$0 = \det(X_u, X_{uu}, X_{uuu}) =
%	\begin{vmatrix}
%		1	&	\omega_u		&	\omega_{uu} + \omega_u^2 - \omega_v^2 - e^{-2\omega} \\
%		0	&	-\omega_v	&	-2\omega_u \omega_v - \omega_{uv}\\
%		0	&	-1			&	-\omega_u
%	\end{vmatrix} =
	-e^{2\omega}(\omega_{uv} + \omega_u \omega_v).$$
Similarly, \eqref{eqn:min_gaussW} implies that a $v$-curvature line is planar if and only if $ \omega_{uv} + \omega_u \omega_v = 0$.
\end{proof}
\begin{rema}
It should be noted that the condition $\omega_{uv} + \omega_u \omega_v = 0$ is equivalent to the condition $(e^\omega)_{uv} = 0$ as found in \cite{barthel1980} or \cite{blaschke1923}. Since the curvature lines of the original minimal surface correspond to asymptotic lines of its conjugate surface, the above equivalence shows that the conjugate of the minimal surface with planar curvature lines is an affine minimal surface (\cite{barthel1980}, \cite{blaschke1923}, \cite{schaal1973}, \cite{thomsen1923}).
\end{rema}

Hence, finding non-planar umbilic-free minimal surfaces with planar curvature lines is equivalent to finding solutions to the following system of partial differential equations:
	\begin{subnumcases}{\label{eqn:min_pde}}
		\Delta \omega - e^{-2\omega} = 0 \label{eqn:min_pde1} &\text(minimality condition)\\
		\omega_{uv} + \omega_u \omega_v = 0 \label{eqn:min_pde2} &\text(planar curvature line condition).
	\end{subnumcases}
Generally, solving systems of partial differential equations may prove to be difficult. However, the next lemma shows that \eqref{eqn:min_pde} can be reduced to a system of ordinary differential equations.

\begin{lemm}\label{lemm:min_solutionOmega}
The solution $\omega : \Sigma \to \mathbb{R}$ of \eqref{eqn:min_pde} is precisely given by
\begin{equation} \label{eqn:min_solutionOmega}
	e^{\omega(u,v)} = \frac{1 + f(u)^2 + g(v)^2}{f_u(u) + g_v(v)}
\end{equation}
where $f(u)$ and $g(v)$ are real-valued meromorphic functions satisfying the following system of ordinary differential equations:
	\begin{subnumcases}{\label{eqn:min_ode}}
		(f_u(u))^2 = (c - d) f(u)^2 + c \label{eqn:min_ode1}\\
		f_{uu} (u)= (c - d) f(u) \label{eqn:min_ode2}\\
		(g_v(v))^2 = (d - c) g(v)^2 + d \label{eqn:min_ode3}\\
		g_{vv}(v) = (d - c) g(v). \label{eqn:min_ode4}
	\end{subnumcases}
for some real constants $c$ and $d$ such that $c^2 + d^2 \neq 0$. Moreover, $f$ and $g$ can be recovered from $\omega$ by
\begin{equation} \label{eqn:min_recover}
	\begin{cases}
		\omega_u = e^{-\omega} f(u)\\
		\omega_v = e^{-\omega} g(v)
	\end{cases}.
\end{equation}
\end{lemm}

\begin{proof}
Integrating \eqref{eqn:min_pde2} with respect to $u$ and $v$ gives \eqref{eqn:min_recover} for some constants of integration $f(u)$ and $g(v)$. Using these definitions of $f$ and $g$, it is straightforward to check that \eqref{eqn:min_solutionOmega} holds.

Now, from the fact that $\omega_u e^{-\omega} = e^{-2\omega} f$,
\begin{equation} \label{eqn:min_tempOmega1}
	\frac{f_{uu}}{1 + f^2 + g^2} - \frac{f((f_u)^2 - (g_v)^2)}{1 + f^2 + g^2} = 0,
\end{equation}
and multiplying both sides by $2f_u$ and intergrating with respect to $u$ tells us that
\begin{equation} \label{eqn:min_tempOmega2}
	(f_u)^2 = (g_v)^2 + D(v)(1 + f^2 + g^2)
\end{equation}
for some constant of integration $D(v)$. Substituting \eqref{eqn:min_tempOmega2} into \eqref{eqn:min_tempOmega1}, we get $f_{uu}(u) = D(v) f(u)$ implying that $D(v) = \tilde{c}$ for some constant $\tilde{c}$. Hence,
$$f_{uu} = \tilde{c}f,$$
and again multiplying both sides by $2f_u$ and intergrating with respect to $u$ implies that
$$(f_u)^2 = \tilde{c}f^2 + c$$
for some constant $c$.

Similarly, from the fact that  $\omega_v e^{-\omega} = e^{-2\omega} g$, we can show that
$$\begin{cases}
	g_{vv} = \tilde{d}g\\
	(g_v)^2 = \tilde{d}g^2 + d
\end{cases}$$
for some constants $d$ and $\tilde{d}$.
Substituting these differential equations into \eqref{eqn:min_tempOmega1} shows that $-\tilde{d} = \tilde{c} = c - d$.
\end{proof}

To find the explicit solution for $f$, we must consider the initial conditions of $f(u)$ and $g(v)$ satisfying \eqref{eqn:min_ode}. We would like to assume $f(0) = g(0) = 0$ for simplicity; therefore, we first identify the conditions for $f(u)$ and $g(v)$ having a zero and prove that both $f(u)$ and $g(v)$ has a zero, using the next pair of lemmas.

\begin{lemm}\label{lemm:min_fzeroCondition}
$f(u)$ (resp. $g(v)$) satisfying Lemma \ref{lemm:min_solutionOmega} has a zero if and only if $c \geq 0$ (resp. $d \geq 0$).
\end{lemm}
\begin{proof}
If $c = d$, then the statement is a result of direct computation. Now, assume $c \not = d$. First, to see that $c \geq 0$ implies that $f(u)$ has a zero, from \eqref{eqn:min_ode1} and \eqref{eqn:min_ode2}, we get
$$f(u) = C_1e^{\sqrt{c-d}\,u} + C_2e^{-\sqrt{c-d}\,u}$$
for some complex constants $C_1$ and $C_2$ such that $c = -4C_1C_2(c-d)$. If $c = 0$, then either $f(u) \equiv 0$ or \eqref{eqn:min_ode1} implies $d < 0$, a contradiction to \eqref{eqn:min_ode3}. Now assume $c > 0$. Then, $C_1$ and $C_2$ are non-zero, and we may let $C_1 = \frac{1}{2}\sqrt{\frac{c}{c-d}} = -C_2$ so that $f$ is real-valued and $f(0) = 0$ .

On the other hand, if $f(u_0) = 0$ for some $u_0$, then $(f_u(u_0))^2 = c$, implying that $c \geq 0$. Therefore, $f(u)$ has a zero if and only if $c \geq 0$. The case for $g(v)$ is proven similarly using \eqref{eqn:min_ode4}.
\end{proof}

\begin{lemm}\label{lemm:min_fgzero}
Let $f(u)$ and $g(v)$ be functions satisfying \eqref{eqn:min_ode}. Then both $f(u)$ and $g(v)$ have a zero.
\end{lemm}
\begin{proof}
Suppose by way of contradiction that  $f$ does not have a zero. Then by the previous lemma, $c < 0$. From \eqref{eqn:min_ode1}, we see that $c < 0$ implies $c - d > 0$. However, \eqref{eqn:min_ode3} implies that if $c - d > 0$, then $d > 0$, a contradiction since $c < 0$ and $c > d$. Therefore, $f$ must always have a zero. Similarly, $g$ must have a zero, from \eqref{eqn:min_ode2} and \eqref{eqn:min_ode4}.
\end{proof}

By shifting parameters $u$ and $v$, we may assume $f(0) = g(0) = 0$. Using these initial conditions, we may solve \eqref{eqn:min_ode} to get the following:
\begin{equation} \label{eqn:min_explicitfgPM}
	\begin{aligned}
		f(u) &=
			\begin{cases}
				\pm \displaystyle\frac{\alpha}{\sqrt{\alpha^2 - \beta^2}} \sinh{(\sqrt{\alpha^2 - \beta^2}\,u)}, &\text{if }\alpha \neq \beta \\
				\pm \alpha u, &\text{if }\alpha = \beta
			\end{cases}\\
		g(v) &=
            		\begin{cases}
            			\pm \displaystyle\frac{\beta}{\sqrt{\beta^2 - \alpha^2}} \sinh{(\sqrt{\beta^2 - \alpha^2}\,v)}, &\text{if }\alpha \neq \beta \\
            			\pm \beta v, &\text{if }\alpha = \beta
            		\end{cases}
	\end{aligned}
\end{equation}
where $\alpha^2 = c$ and $\beta^2 = d$. It should be noted that by letting $u \mapsto -u$ and $v \mapsto -v$, we may drop the plus or minus condition of \eqref{eqn:min_explicitfgPM}.
Finally, we arrive at the following result.

\begin{prop}\label{prop:min_solution}
For non-planar minimal surface $X(u,v)$ with planar curvature lines, the real-analytic solution $\omega : \mathbb{R}^2 \to \mathbb{R}$ of \eqref{eqn:min_pde} is precisely given by
\begin{equation}\label{eqn:omega}
	e^{\omega(u,v)} = \frac{1 + f(u)^2 + g(v)^2}{f_u(u) + g_v(v)}
\end{equation}
with
\begin{equation} \label{eqn:min_explicitfg}
	\begin{aligned}
		f(u) &=
			\begin{cases}
				\displaystyle\frac{\alpha}{\sqrt{\alpha^2 - \beta^2}} \sinh{(\sqrt{\alpha^2 - \beta^2}\,u)}, &\text{if }\alpha \neq \beta \\
				\alpha u, &\text{if }\alpha = \beta
			\end{cases}\\
		g(v) &=
            		\begin{cases}
            			\displaystyle\frac{\beta}{\sqrt{\beta^2 - \alpha^2}} \sinh{(\sqrt{\beta^2 - \alpha^2}\,v)}, &\text{if }\alpha \neq \beta \\
            			\beta v, &\text{if }\alpha = \beta
            		\end{cases}
	\end{aligned}
\end{equation}
where $\alpha + \beta > 0$.
\end{prop}

\begin{proof}
To see that the $\omega$ in \eqref{eqn:omega} with $f(u)$ and $g(v)$ as in \eqref{eqn:min_explicitfg} is real, we only need to show that $f_u(u) + g_v(v) > 0$ for any $(u,v) \in \Sigma$. If $\alpha = \beta$, then $f_u + g_v = \alpha + \beta > 0$. Without loss of generality, assume $\alpha > \beta$; since $\alpha + \beta > 0$, $\alpha > |\beta|$. From \eqref{eqn:min_explicitfg},
\begin{align*}
	f_u(u) &= \alpha \cosh(\sqrt{\alpha^2 - \beta^2}\,u) \geq \alpha\\
	g_v(v) &= \beta \cos(\sqrt{\alpha^2 - \beta^2}\,v) \geq -|\beta|
\end{align*}
implying that $f_u + g_v \geq \alpha - |\beta| > 0$. The case for $\alpha < \beta$ can be proved similarly. Finally, the real-analyticity of $f(u)$ and $g(v)$ tells us that the domain of $\omega(u,v)$ can be extended to $\mathbb{R}^2$ globally.

\end{proof}

Since the $u$-direction and $v$-direction of $\omega(u,v)$ depend only on $f(u)$ and $g(v)$ respectively, by choosing different values for $\alpha$ and $\beta$, we may analytically understand how the surfaces behave in either direction. The following theorem and figure explains the relationship between different values of $\alpha$ and $\beta$ and the surface generated by the corresponding $\omega(u,v)$. Note that in the figure, the subscript $u \leftrightarrow v$ denotes that the role of $u$ and $v$ are switched.

\begin{theo}\label{theo:abresch}
Let $X(u,v)$ be a non-planar minimal surface in $\mathbb{R}^3$ with isothermic coordinates $(u,v)$ such that $\diff s^2 = e^{2\omega}(\diff u^2 + \diff v^2)$. Then $X$ has planar curvature lines if and only if $\omega(u,v)$ satisfies Proposition \ref{prop:min_solution}. Furthermore, for different values of $\alpha$ and $\beta$, the metric function of $X(u,v)$ have the following properties, based on \Cref{fig:min_bifurcation}:
\begin{itemize}
	\item $\circlenum{1}$, $\circlenum{1}'$ are not periodic in the $u$-direction but periodic in the $v$-direction.
	\item $\circlenum{2}$ is not periodic in the $u$-direction but constant in the $v$-direction.
	\item $\circlenum{3}$ is not periodic in both the $u$-direction and $v$-direction.
\end{itemize}
\end{theo}

\begin{figure}[H]
	\scalebox{0.8}{\includegraphics{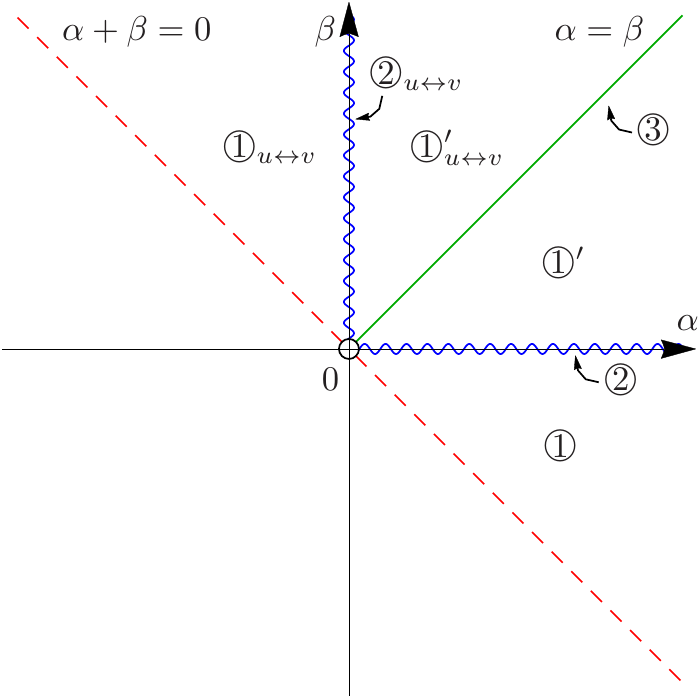}}
	\caption{Classification diagram for non-planar minimal surfaces with planar curvature lines.}
	 \label{fig:min_bifurcation}
\end{figure}

\subsection{Axial directions and normal vector}\label{subsec:min_Walter}
Through Theorem \ref{theo:abresch}, we were able to identify the non-planar minimal surfaces with planar curvature lines. In fact, we may even understand that the surfaces represented by $\circlenum{1}$, $\circlenum{2}$, or $\circlenum{3}$ in \Cref{fig:min_bifurcation} are surfaces in the Bonnet family, catenoid, or Enneper surface, respectively. However, we would like to find their parametrizations, allowing us to visualize these surfaces and obtain a deformation between them. To do this, we utilize the Weierstrass representation theorem as follows:  find the axial directions of the non-planar minimal surfaces with planar curvature lines by referring to the method developed by Walter in \cite{walter1987}, calculate the unit normal vector, and recover the Weierstrass data from the metric function $e^{\omega(u,v)}$ in Proposition \ref{prop:min_solution}. First, we show the existence of axial directions for non-planar minimal surfaces with planar curvature lines.

\begin{prop}
If $f(u)$ (resp. $g(v)$) is not identically equal to zero, then there is a unique constant direction $\vec{v}_1$  (resp. $\vec{v}_2$) such that
\begin{equation}\label{eqn:v1}
	\langle m(u,v), \vec{v}_1\rangle = \langle m_v(u,v), \vec{v}_1 \rangle = 0 \quad \text{(resp. }\langle n(u,v), \vec{v}_2\rangle = \langle n_u(u,v), \vec{v}_2 \rangle = 0\text{)}
\end{equation}
where $m = e^{-2\omega} (X_u \times X_{uu})$ (resp. $n = e^{-2\omega} (X_v \times X_{vv})$) and
$$\vec{v}_1 = \omega_{uu} X_u - \omega_{uv} X_v + \omega_u N \quad \text{(resp. }\vec{v}_2 = \omega_{uv} X_u - \omega_{vv} X_v + \omega_v N\text{).}$$
Furthermore, if $\vec{v}_1$ and $\vec{v}_2$ both exist, then $\vec{v}_1$ is orthogonal to $\vec{v}_2$. We call $\vec{v}_1$ and $\vec{v}_2$ the \emph{axial directions} of the surface.
\end{prop}
\begin{proof}
We will only prove the statement regarding $\vec{v}_1$. From \eqref{eqn:min_gaussW},
$$m = e^{-2\omega} X_v - \omega_v N \quad\text{and}\quad m_v = -\omega_u e^{-2\omega} X_u + \omega_{uu} N.$$
Since $f(u)$ is not identically equal to zero, let $f(u_0) \neq 0$. Then since $\omega_u(u_0,v) \neq 0$, $m$ and $m_v$ are linearly independent at $(u_0,v)$. Therefore, the following definition
$$\vec{v}_1 := e^{-2\omega}(m \times m_v) = \omega_{uu} X_u - \omega_{uv} X_v + \omega_u N$$
is well-defined on $(u_0, v)$. Furthermore, by calculation,
$$(\vec{v}_1)_u = (\vec{v}_1)_v = 0$$
for all $(u, v)$, implying that $\vec{v}_1$ is constant. It is easy to check that \eqref{eqn:v1} holds, and checking orthogonality is a straight-forward calculation using the definitions of $\vec{v}_1$ and $\vec{v}_2$.
\end{proof}

%\begin{rema}
%For any fixed $v_0$, let $P(v_0)$ be the plane including the curvature line $X(u, v_0)$. From the definition in the above proposition, $m(u,v_0)$ is perpendicular to $P(v_0)$; therefore, $\vec{v}_1(u,v_0) \in P(v_0)$.
%\end{rema}

By normalizing these vectors, we may calculate the unit normal vector of the surface as follows.
\begin{prop}
Let $f(u)$ and $g(v)$ be as in Proposition \ref{prop:min_solution}. If $\alpha\beta \not = 0$, then the unit normal vector $N(u,v)$ is given by the following:
$$N(u,v) = \left(\frac{1}{\alpha}\,\omega_u, \frac{1}{\beta}\,\omega_v, \sqrt{1-\frac{1}{\alpha^2}\,\omega_u^2- \frac{1}{\beta^2}\,\omega_v^2}\right).$$
\end{prop}
\begin{proof}
First we normalize the axial direction and set $\vec{v}_1 = \mathbf{e}_1$, the unit vector in the $x_1$-direction. Then, from $\langle m, \vec{v}_1 \rangle = 0$, we get
\begin{equation} \label{eqn:min_normal1}
	-\omega_v N_1 - (N_1)_v = 0,
\end{equation}
where $N = (N_1, N_2, N_3)$. Now, using \eqref{eqn:min_pde2} and integrating with respect to $v$, we obtain
$$N_1 = B_1(u)\cdot\omega_u$$
for some function $B_1(u)$. Similarly, from $\langle m_v, \vec{v}_1 \rangle = 0$, we get
$$N_1 = B_2(v)\cdot\omega_u$$
for some function $B_2(v)$.
Therefore, $B_1(u) = B_2(v) = B$ for some constant $B$, and
$$N_1 = B\cdot\omega_u.$$

To compute $B$, first note that since $\omega_u(0,v) = e^{\omega(0,v)}f(0) = 0$, we note that $\langle m, X_u \rangle = \langle m_v, X_u \rangle = 0$ on $(0, v)$. Therefore, $X_u(0,v) \in \spann\{\mathbf e_1\}$, and
$$e^{2\omega} = \|X_u(0,v)\|^2 = ((X_1(0,v))_v)^2 = e^{2\omega}B^2\alpha^2.$$
Hence,
$$B = \frac{1}{\alpha} \quad\text{and}\quad N_1 = \frac{1}{\alpha}\,\omega_u.$$

Similarly, by setting $\vec{v}_2 = \mathbf{e}_2$, we get
$$N_2 = \frac{1}{\beta}\,\omega_v.$$
Finally, using the fact that $N$ is a unit normal vector gives us the desired result.
\end{proof}

Using the normal vector, we may now calculate the Weierstrass data. Since the meromorphic function $h$ is the normal vector function under stereographic projection,
$$h^{(\alpha,\beta)}(u,v) =  \frac{1}{1-N_3}\,(N_1 + i N_2)  = \frac{\sqrt{\alpha^2-\beta^2}}{\alpha - \beta}\tanh{\left(\frac{\sqrt{\alpha^2-\beta^2}}{2}\,(u + i v)\right)},$$
while since $Q = -\frac{1}{2}(h_u + i h_v)\eta = -\frac{1}{2}$, we also have
$$\eta^{(\alpha,\beta)}(u,v) = \frac{1}{h_u + i h_v} = \frac{1}{\alpha + \beta} \cosh^2{\left(\frac{\sqrt{\alpha^2 - \beta^2}}{2}(u + i v)\right)}$$
for $\alpha + \beta > 0$. If we let $\alpha = r \cos\theta$ and $\beta = r \sin\theta$, then it is easy to see that $r$ is a homothety factor. Therefore, we may assume $r = 1$, and rewrite $h^{(\alpha,\beta)}(u,v)$ and $\eta^{(\alpha,\beta)}(u,v)$ as follows:
\begin{equation}\label{eqn:min_w-data}
	\begin{aligned}
		h^{\theta}(u,v) &=
			\begin{cases}
				\dfrac{\sqrt{\cos{(2\theta)}}}{\cos\theta - \sin\theta} \tanh{\left(\dfrac{\sqrt{\cos{(2\theta)}}}{2}(u + i v)\right)}, &\text{if } \theta \neq \frac{\pi}{4} \\[12 pt]
				\dfrac{u + i v}{\sqrt{2}}, &\text{if } \theta = \frac{\pi}{4}
			\end{cases}\\
		\eta^{\theta}(u,v) &=
			\begin{cases}
				\dfrac{1}{\cos\theta + \sin\theta}\cosh^2\left(\dfrac{\sqrt{\cos{(2\theta)}}}{2}(u + i v)\right), &\text{if } \theta \neq \frac{\pi}{4} \\[12 pt]
				\dfrac{1}{\sqrt{2}}, &\text{if } \theta = \frac{\pi}{4}
			\end{cases}
	\end{aligned}
\end{equation}
where $\theta \in \left(-\frac{\pi}{4},\frac{3\pi}{4}\right)$. Since $h^\theta$ is meromorphic, and $\eta^\theta$ is holomorphic such that $(h^\theta)^2\eta^\theta$ is holomorphic, we may use the Weierstrass representation theorem to obtain the following parametrizations for minimal surfaces with planar curvature lines.
\begin{prop}
Let $X(u,v)$ be a non-planar minimal surface with planar curvature lines in $\mathbb{R}^3$. Then $X$ must have the following parametrization
\begin{equation}\label{eqn:min_parametrization}
	X^{\theta}(u,v) =
    	\begin{cases}
    		\begin{pmatrix}
                		\dfrac{u \cos{\theta}\sqrt{\cos{2\theta}} - \sin{\theta}\sinh{(u\sqrt{\cos{2\theta}})}\cos{(v\sqrt{\cos{2\theta}}})}{(\cos{2\theta})^{3/2}} \\[9pt]
                		\dfrac{v \sin{\theta}\sqrt{\cos{2\theta}} - \cos{\theta}\cosh{(u\sqrt{\cos{2\theta}}})\sin{(v\sqrt{\cos{2\theta}})}}{(\cos{2\theta})^{3/2}} \\[9pt]
                		\dfrac{\cosh{(u\sqrt{\cos{2\theta}})}\cos{(v\sqrt{\cos{2\theta}}})-1}{\cos{2\theta}}\\
                	\end{pmatrix}^t, &\text{if }\theta \neq \frac{\pi}{4} \\[40pt]
		\left(-\dfrac{u (-6 + u^2 - 3v^2)}{6\sqrt{2}}, \dfrac{v(-6 - 3u^2 + v^2)}{6\sqrt{2}}, \dfrac{u^2 - v^2}{2}\right), &\text{if }\theta = \frac{\pi}{4}
    	\end{cases}
\end{equation}
for some $\theta \in \left(-\frac{\pi}{4},\frac{3\pi}{4}\right)$ on its domain up to isometries and homotheties.
\end{prop}

\section{Continuous deformation of minimal surfaces with planar curvature lines} \label{sec:deformation}
In the previous section, we obtained the Weierstrass data and the parametrizations of non-planar minimal surfaces in $\mathbb{R}^3$ with planar curvature lines; in fact, the Weierstrass data and the parametrizations of such surfaces depended on a single parameter $\theta$. In this section, we show that this parameter defines a locally continuous deformation between non-planar minimal surfaces preserving the planar curvature line condition. Furthermore, we show that by introducing a suitable homothety factor depending on $\theta$, we may also extend the deformation to include the plane.

First, we show that the deformation of minimal surfaces in $\mathbb{R}^3$ with planar curvature lines we obtain by using the parameter $\theta$ is continuous. By ``continuous'', we mean that the deformation converges uniformally over compact subdomains component-wise. To show this, it is enough to show that each component function in the parametrization is continuous for all $\theta$ at any point $(u,v)$ in the domain. The continuity is self-evident at any $\theta \neq \frac{\pi}{4}$; hence, we only need to check the for the case $\theta = \frac{\pi}{4}$.

However, for the Weierstrass data of minimal surfaces with planar curvature lines as stated in \eqref{eqn:min_w-data}, it is easy to check that at any point $(u,v)$,
$$\lim_{\theta \to \frac{\pi}{4}} h^\theta(u,v) = h^{\frac{\pi}{4}}(u,v) \quad\text{and}\quad \lim_{\theta \to \frac{\pi}{4}} \eta^\theta(u,v) = \eta^{\frac{\pi}{4}}(u,v).$$
In addition, each component of the parametrization in \eqref{eqn:min_parametrization} is also continuous at $\theta = \frac{\pi}{4}$ at any point $(u,v)$, as
$$\lim_{\theta \to \frac{\pi}{4}} X^\theta(u,v) = X^{\frac{\pi}{4}}(u,v).$$
Therefore, $X^\theta(u,v)$ is a continuous deformation.

Now, we would like to extend $X^\theta(u,v)$ to include the plane. To do so, we define the homotethy factor $R^\theta$ as follows:
$$R^\theta = \left(1-\sin\left(\theta + \frac{\pi}{4}\right)\right)|\cos 2\theta| + \sin\left(\theta + \frac{\pi}{4}\right)$$
for $\theta \in [-\frac{\pi}{4},\frac{3\pi}{4}]$. Note that $R^\theta \geq 0$, and $R^\theta = 0$ if and only if $\theta$ is at an endpoint of its domain.

If we consider
$$\tilde X^\theta(u,v) = R^\theta X^\theta(u,v),$$
then
$$\lim_{\theta \searrow -\frac{\pi}{4}} \tilde X^\theta(u,v) = \frac{3}{\sqrt{2}}(u, -v, 0) = \lim_{\theta \nearrow \frac{3\pi}{4}} \tilde X^\theta(u,v).$$
Therefore, the extension of $X^\theta(u, v)$ to $\tilde X^\theta(u,v)$ defined as
$$\tilde X^\theta (u,v) =
	\begin{cases}
		\frac{3}{\sqrt{2}}(u, -v, 0), &\text{if }\theta = -\frac{\pi}{4}, \frac{3\pi}{4}\\
		R^\theta X^\theta(u,v), &\text{if }\theta \in \left(-\frac{\pi}{4}, \frac{3\pi}{4}\right)
	\end{cases}$$
is again, a continuous deformation for $\theta \in [-\frac{\pi}{4},\frac{3\pi}{4}]$.

In conclusion, we obtain the following classification and deformation of minimal surfaces with planar curvature lines.
\begin{theo}\label{theo:min_main}
If $\tilde X(u,v)$ is a minimal surface with planar curvature lines in $\mathbb{R}^3$, then the surface is given by the following parametrization on its domain
$$
	\tilde X^{\theta}(u,v) =
    	\begin{cases}
    		R^\theta\begin{pmatrix}
                		\dfrac{u \cos{\theta}\sqrt{\cos{2\theta}} - \sin{\theta}\sinh{(u\sqrt{\cos{2\theta}})}\cos{(v\sqrt{\cos{2\theta}}})}{(\cos{2\theta})^{3/2}} \\[9pt]
                		\dfrac{v \sin{\theta}\sqrt{\cos{2\theta}} - \cos{\theta}\cosh{(u\sqrt{\cos{2\theta}}})\sin{(v\sqrt{\cos{2\theta}})}}{(\cos{2\theta})^{3/2}} \\[9pt]
                		\dfrac{\cosh{(u\sqrt{\cos{2\theta}})}\cos{(v\sqrt{\cos{2\theta}}})-1}{\cos{2\theta}}\\
                	\end{pmatrix}^t, &\text{if }\theta \neq -\frac{\pi}{4},\frac{\pi}{4},\frac{3\pi}{4} \\[37pt]
		\left(-\dfrac{u (-6 + u^2 - 3v^2)}{6\sqrt{2}}, \dfrac{v(-6 - 3u^2 + v^2)}{6\sqrt{2}}, \dfrac{u^2 - v^2}{2}\right), &\text{if }\theta = \frac{\pi}{4} \\[9pt]
		\dfrac{3}{\sqrt{2}}(u, -v, 0), &\text{if }\theta = -\frac{\pi}{4}, \frac{3\pi}{4}
    	\end{cases}
$$
up to isometries and homotheties of $\mathbb{R}^3$ for some $\theta \in \left[-\frac{\pi}{4},\frac{3\pi}{4}\right]$, where $R^\theta = \left(1-\sin\left(\theta + \frac{\pi}{4}\right)\right)|\cos 2\theta| + \sin\left(\theta + \frac{\pi}{4}\right)$.
In fact, it must be a piece of one, and only one, of the following:
\begin{itemize}
	\item plane $(\theta = -\frac{\pi}{4}, \frac{3\pi}{4})$,
	\item catenoid $(\theta = 0, \frac{\pi}{2})$,
	\item Enneper's surface $(\theta = \frac{\pi}{4})$, or
	\item a surface in the Bonnet family $(\theta \in (-\frac{\pi}{4}, \frac{3\pi}{4})\setminus \{0, \frac{\pi}{4}, \frac{\pi}{2}\})$.
\end{itemize}
Moreover, the deformation $\tilde X^\theta(u,v)$ depending on the parameter $\theta$ is continuous (see \Cref{fig:deformation1}).
\end{theo}

\begin{figure}[H]
	\scalebox{0.9}{\includegraphics{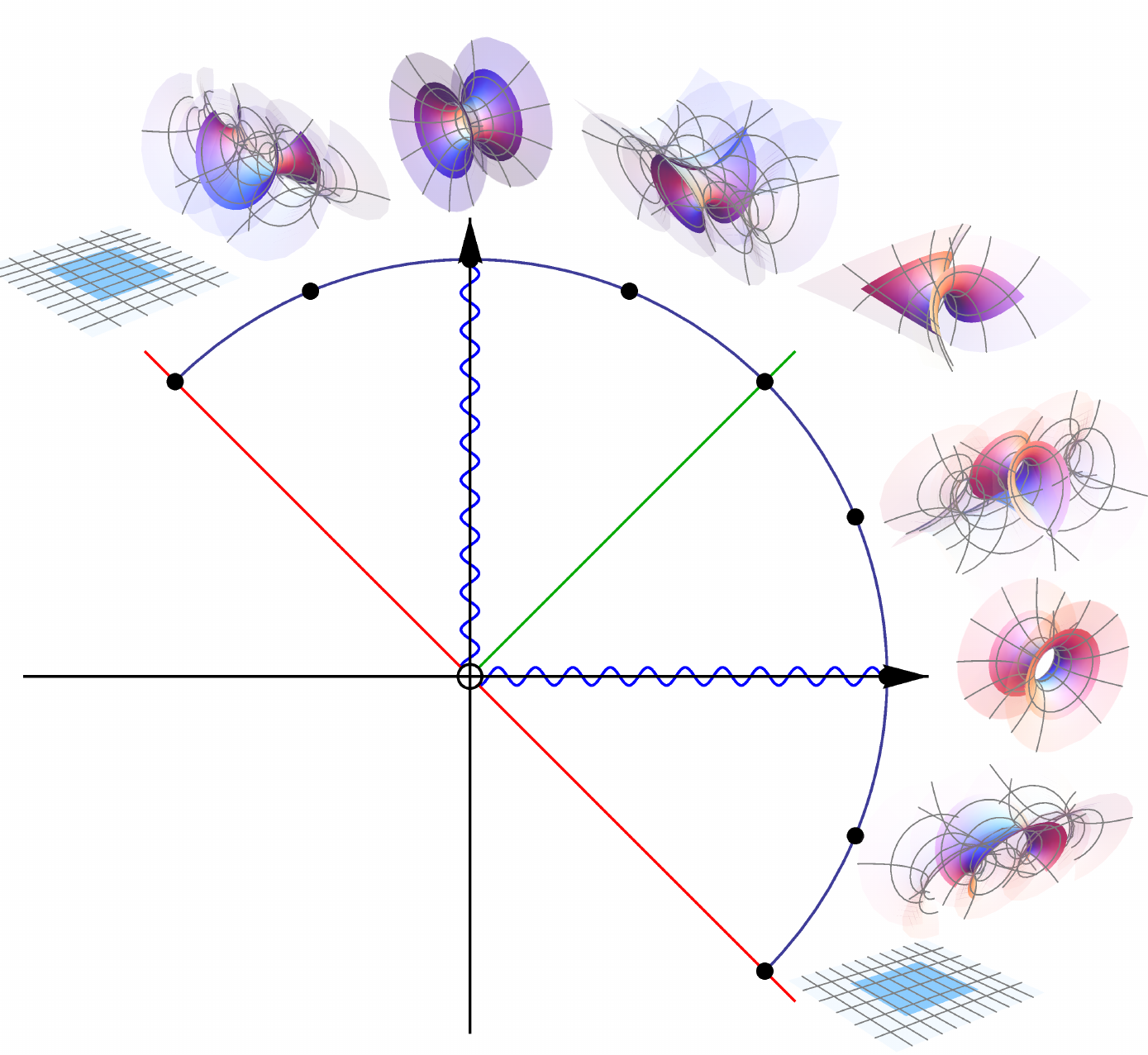}}
	\caption{Deformation of minimal surfaces with planar curvature lines with parametrization as in Theorem \ref{theo:min_main}.}
	\label{fig:deformation1}
\end{figure}

Furthermore, by considering the conjugate of minimal surfaces with planar curvature lines, we get the following classification and deformation of minimal surfaces that are also affine minimal.

\begin{coro}\label{coro:thomsen}
If $\hat X(u,v)$ is a minimal surface that is also an affine minimal surface in $\mathbb{R}^3$, then the surface is given by the following parametrization on its domain
$$
	\hat X^{\theta}(u,v) =
    	\begin{cases}
    		R^\theta\begin{pmatrix}
                		\dfrac{-u \cos{\theta}\sqrt{\cos{2\theta}} + \sin{\theta}\cosh{(u\sqrt{\cos{2\theta}})}\sin{(v\sqrt{\cos{2\theta}}})}{(\cos{2\theta})^{3/2}} \\[9pt]
                		\dfrac{v \sin{\theta}\sqrt{\cos{2\theta}} - \cos{\theta}\sinh{(u\sqrt{\cos{2\theta}}})\cos{(v\sqrt{\cos{2\theta}})}}{(\cos{2\theta})^{3/2}} \\[9pt]
                		\dfrac{-\sinh{(u\sqrt{\cos{2\theta}})}\sin{(v\sqrt{\cos{2\theta}}})}{\cos{2\theta}}\\
                	\end{pmatrix}^t, &\text{if }\theta \neq -\frac{\pi}{4},\frac{\pi}{4},\frac{3\pi}{4} \\[37pt]
		\left(-\dfrac{v (6 - 3u^2 + v^2)}{6\sqrt{2}}, \dfrac{u(6 + u^2 - 3v^2)}{6\sqrt{2}}, -uv\right), &\text{if }\theta = \frac{\pi}{4} \\[9pt]
		\dfrac{3}{\sqrt{2}}(-v, -u, 0), &\text{if }\theta = -\frac{\pi}{4}, \frac{3\pi}{4}
    	\end{cases}
$$
In fact, it must be a piece of one, and only one, of the following:
\begin{itemize}
	\item plane $(\theta = -\frac{\pi}{4}, \frac{3\pi}{4})$,
	\item helicoid $(\theta = 0, \frac{\pi}{2})$,
	\item Enneper surface $(\theta = \frac{\pi}{4})$, or
	\item a surface in the Thomsen family $(\theta \in (-\frac{\pi}{4}, \frac{3\pi}{4})\setminus \{0, \frac{\pi}{4}, \frac{\pi}{2}\})$.
\end{itemize}
up to isometries and homotheties of $\mathbb{R}^3$ for some $\theta \in \left[-\frac{\pi}{4},\frac{3\pi}{4}\right]$.

Moreover, the deformation $\hat X^\theta(u,v)$ depending on the parameter $\theta$ is continuous (see \Cref{fig:deformation2}).
\end{coro}

\begin{figure}[H]
	\scalebox{1}{\includegraphics{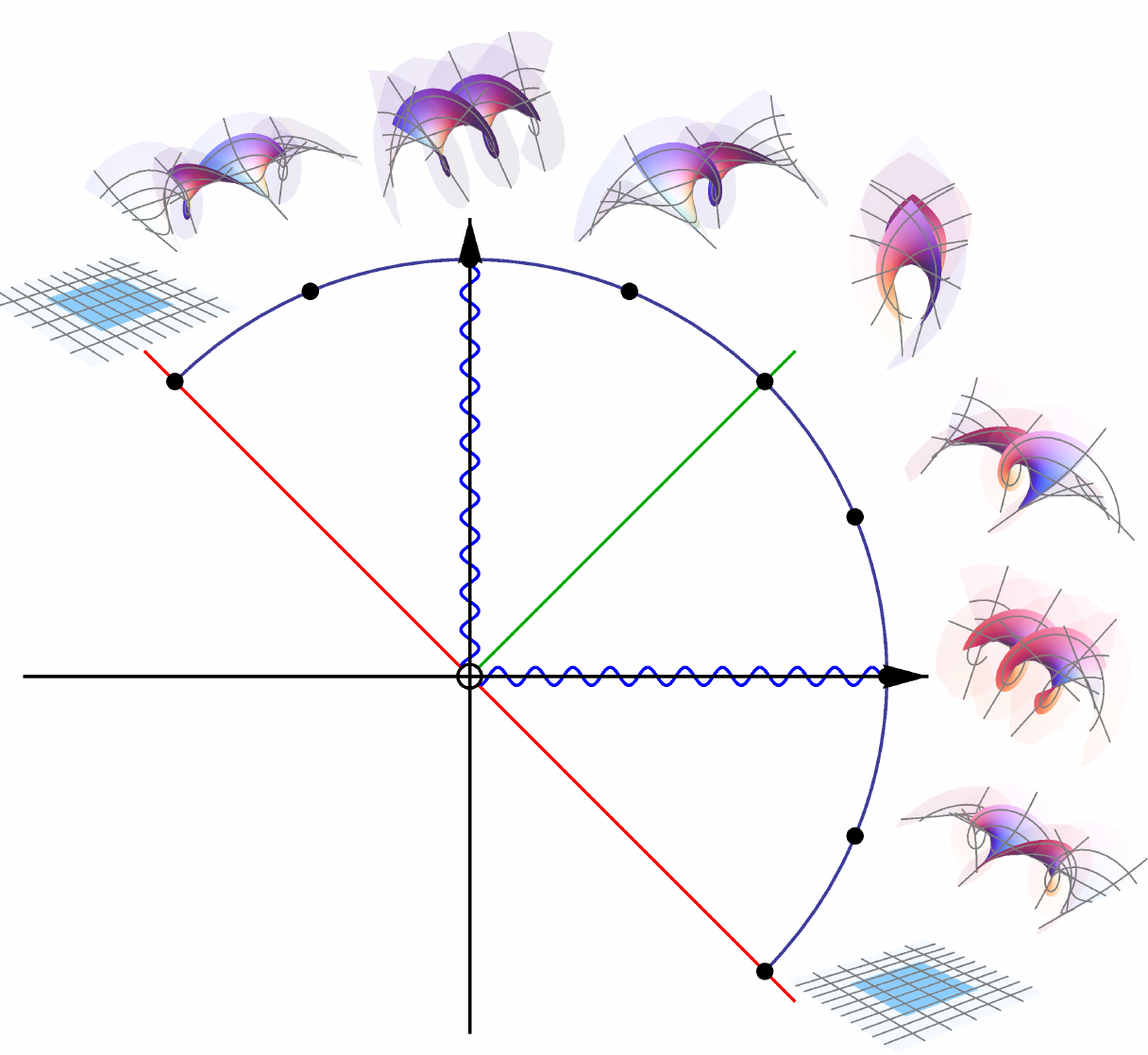}}
	\caption{Deformation of minimal surfaces that are also affine minimal with parametrization as in the corollary to Theorem \ref{theo:min_main}.}
	\label{fig:deformation2}
\end{figure}

\section{Final remarks}
With appropriate modifications, the analytic method in conjunction with axial directions introduced in this paper can be applied to study maximal surfaces with planar curvature lines in Lorentz-Minkowski space $\mathbb{R}^{2,1}$. The complete classification is already given by Leite in \cite{leite2015}, where she developed and used the orthogonal systems of circles on the hyperbolic plane. The analytic method can also be used to fully classify such surfaces; however, using the analytic method further allows us to see that all maximal surfaces with planar curvature lines can be joined by a single deformation. In addition, using this method allows us to understand that maximal Bonnet-type surfaces can be classified into three general types analytically, and five specific types by means of singularity theory. We will defer a complete discussion of these properties to our subsequent work \cite{ogata2016}.

\vspace{15pt}
{\bf Acknowledgements.} The authors express their gratitude to Professor Shoichi Fujimori, Professor Hitoshi Furuhata, and Professor Wayne Rossman for many useful comments.

\nocite{*}

\bibliography{Joseph-Ogata}{}

\begin{thebibliography}{10}

\bibitem{abresch1987}
U.~Abresch.
\newblock Constant mean curvature tori in terms of elliptic functions.
\newblock {\em J. Reine Angew. Math.}, 374:169--192, 1987.

\bibitem{barthel1980}
W.~Barthel, R.~Volkmer, and I.~Haubitz.
\newblock Thomsensche {Minimalflächen}—analytisch und anschaulich.
\newblock {\em Resultate Math.}, 3(2):129--154, 1980.

\bibitem{blaschke1923}
W.~Blaschke.
\newblock {\em Vorlesungen über {Differentialgeometrie} und geometrische
  {Grundlagen} von {Einsteins} {Relativitätstheorie} {II}: {Affine}
  {Differentialgeometrie}}.
\newblock Springer, Berlin, 1923.

\bibitem{bonnet1855}
O.~Bonnet.
\newblock Observations sur les surfaces minima.
\newblock {\em C. R. Acad. Sci. Paris}, 41:1057--1058, 1855.

\bibitem{ogata2016}
J.~Cho and Y.~Ogata.
\newblock Deformation and singularities of maximal surfaces with planar
  curvature lines.
\newblock {\em submitted}.

\bibitem{dobriner1887}
H.~Dobriner.
\newblock Die {Minimalflächen} mit {Einem} {System} {Sphärischer}
  {Krümmungslinien}.
\newblock {\em Acta Math.}, 10(1):145--152, 1887.

\bibitem{eisenhart1909}
L.~P. Eisenhart.
\newblock {\em A {Treatise} on the {Differential} {Geometry} of {Curves} and
  {Surfaces}}.
\newblock Ginn and Company, Boston, 1909.

\bibitem{enneper1878}
A.~Enneper.
\newblock Untersuchungen über die {Flächen} mit planen und sphärischen
  {Krümmungslinien}.
\newblock {\em Abh. Königl. Ges. Wissensch. Göttingen}, 23, 1878.

\bibitem{goursat1887}
E.~Goursat.
\newblock Sur un mode de transformation des surfaces minima.
\newblock {\em Acta Math.}, 11(1-4):135--186, 1887.

\bibitem{goursat1887-1}
E.~Goursat.
\newblock Sur un mode de transformation des surfaces minima.
\newblock {\em Acta Math.}, 11(1-4):257--264, 1887.
\newblock Second Mémoire.

\bibitem{hertrich-jeromin2003}
U.~Hertrich-Jeromin.
\newblock {\em Introduction to {Möbius} differential geometry}, volume 300 of
  {\em London {Mathematical} {Society} {Lecture} {Note} {Series}}.
\newblock Cambridge University Press, Cambridge, 2003.

\bibitem{hertrich-jeromin2001}
U.~Hertrich-Jeromin, E.~Musso, and L.~Nicolodi.
\newblock Möbius geometry of surfaces of constant mean curvature 1 in
  hyperbolic space.
\newblock {\em Ann. Global Anal. Geom.}, 19(2):185--205, 2001.

\bibitem{kreft1908}
W.~Kreft.
\newblock {\em Beiträge zur {Goursat}'schen {Transformation} der
  {Minimalflächen}}.
\newblock {PhD} thesis, Westf. Wilhelms-Univ., 1908.

\bibitem{leite2015}
M.~L. Leite.
\newblock Surfaces with planar lines of curvature and orthogonal systems of
  cycles.
\newblock {\em J. Math. Anal. Appl.}, 421(2):1254--1273, 2015.

\bibitem{leschke????}
K.~Leschke and K.~Moriya.
\newblock Simple factor dressing and the {Lopez}-{Ros} deformation of minimal
  surfaces in {Euclidean} 3-space.
\newblock {\em preprint}.
\newblock arXiv: 1409.5286.

\bibitem{lopez1991}
F.~J. López and A.~Ros.
\newblock On embedded complete minimal surfaces of genus zero.
\newblock {\em J. Differential Geom.}, 33(1):293--300, 1991.

\bibitem{mladenov1999}
I.~M. Mladenov.
\newblock Generalization of {Goursat} transformation of the minimal surfaces.
\newblock {\em C. R. Acad. Bulgare Sci.}, 52(5-6):23--26, 1999.

\bibitem{muller2016}
C.~Müller.
\newblock Planar discrete isothermic nets of conical type.
\newblock {\em Beitr. Algebra Geom.}, 57(2):459--482, 2016.

\bibitem{nitsche1989}
J.~C.~C. Nitsche.
\newblock {\em Lectures on {Minimal} {Surfaces}}, volume~1.
\newblock Cambridge University Press, Cambridge, 1989.
\newblock Introduction, fundamentals, geometry and basic boundary value
  problems, Translated from the German by Jerry M. Feinberg, With a German
  foreword.

\bibitem{nomizu1994}
K.~Nomizu and T.~Sasaki.
\newblock {\em Affine {Differential} {Geometry}}, volume 111 of {\em Cambridge
  {Tracts} in {Mathematics}}.
\newblock Cambridge University Press, Cambridge, 1994.
\newblock Geometry of affine immersions.

\bibitem{pabel1994}
H.~Pabel.
\newblock Deformationen von {Minimalflächen}.
\newblock In {\em Geometrie und ihre {Anwendungen}}, pages 107--139. Hanser,
  Munich, 1994.

\bibitem{perez2002}
J.~Pérez and A.~Ros.
\newblock Properly embedded minimal surfaces with finite total curvature.
\newblock In {\em The global theory of minimal surfaces in flat spaces
  ({Martina} {Franca}, 1999)}, volume 1775 of {\em Lecture {Notes} in {Math}.},
  pages 15--66. Springer, Berlin, 2002.

\bibitem{schaal1973}
H.~Schaal.
\newblock Die {Ennepersche} {Minimalfläche} als {Grenzfall} der
  {Minimalfläche} von {G}. {Thomsen}.
\newblock {\em Arch. Math. (Basel)}, 24:320--322, 1973.

\bibitem{thomsen1923}
G.~Thomsen.
\newblock Uber affine {Geometrie} {XXXIX}.
\newblock {\em Abh. Math. Sem. Univ. Hamburg}, 2(1):71--73, 1923.

\bibitem{walter1987}
R.~Walter.
\newblock Explicit examples to the {H}-problem of {Heinz} {Hopf}.
\newblock {\em Geom. Dedicata}, 23(2):187--213, 1987.

\bibitem{weierstrass1866}
K.~T. Weierstrass.
\newblock Untersuchungen über die {Flächen}, deren mittlere {Krümmung}
  überall gleich {Null} ist.
\newblock {\em Moatsber. Berliner Akad.}, pages 612--625, 1866.

\bibitem{wente1986}
H.~C. Wente.
\newblock Counterexample to a conjecture of {H}. {Hopf}.
\newblock {\em Pacific J. Math.}, 121(1):193--243, 1986.

\bibitem{wente1992}
H.~C. Wente.
\newblock Constant mean curvature immersions of {Enneper} type.
\newblock {\em Mem. Amer. Math. Soc.}, 100(478):vi+77, 1992.

\end{thebibliography}
\bibliographystyle{abbrv}

\end{document}